\documentclass[a4paper, reqno]{amsart}

\usepackage[UKenglish]{babel}
\usepackage{amsmath, amssymb, amsthm}
\usepackage{scrextend}
\usepackage[hidelinks]{hyperref}
\usepackage{xpatch}
\usepackage{color}
\usepackage{tikz-cd}
\usepackage{enumitem}
\usepackage{theoremref}
\usepackage{xspace}

\usepackage[inline,final]{showlabels}
\showlabels{thlabel}

\usepackage{imakeidx}
\makeindex[name=notes,title=Notes throughout the document (red text),columns=1]

\xpatchcmd{\proof}{\itshape}{\normalfont\bfseries}{}{}
\newtheoremstyle{repeat}{}{}{\itshape}{}{\bfseries}{.}{.5em}{#3, repeated}

\newtheorem{theorem}{Theorem}[section]
\newtheorem{proposition}[theorem]{Proposition}
\newtheorem{lemma}[theorem]{Lemma}

\newtheorem{fact}[theorem]{Fact}

\theoremstyle{definition}
\newtheorem{definition}[theorem]{Definition}
\newtheorem{remark}[theorem]{Remark}

\newtheorem{example}[theorem]{Example}

\theoremstyle{repeat}
\newtheorem*{repeated-theorem}{Repeat}

\renewcommand{\L}{\mathcal{L}}
\newcommand{\MM}{\mathfrak{M}}

\DeclareMathOperator{\tp}{tp}
\DeclareMathOperator{\qftp}{qftp}

\DeclareMathOperator{\Aut}{Aut}

\renewcommand{\d}{\operatorname{d}}
\DeclareMathOperator{\EM}{EM}
\DeclareMathOperator{\cl}{cl}
\DeclareMathOperator{\cf}{cf}

\renewcommand{\phi}{\varphi}

\newcommand{\equivsqf}{\equiv^\textup{s-qf}}
\newcommand{\equivstrqf}{\equiv^\textup{str-qf}}
\newcommand{\equivstrOqf}{\equiv^\textup{str$_0$-qf}}
\newcommand{\equivarqf}{\equiv^\textup{ar-qf}}

\newcommand{\lex}{{\textup{lex}}}
\newcommand{\len}{{\textup{len}}}
\newcommand{\TPtwo}{$\mathsf{TP_2}$\xspace}
\newcommand{\s}{{\textup{s}}}
\newcommand{\str}{{\textup{str}}}
\newcommand{\ar}{{\textup{ar}}}

\def\Ind#1#2{#1\setbox0=\hbox{$#1x$}\kern\wd0\hbox to 0pt{\hss$#1\mid$\hss}
\lower.9\ht0\hbox to 0pt{\hss$#1\smile$\hss}\kern\wd0}

\def\Notind#1#2{#1\setbox0=\hbox{$#1x$}\kern\wd0\hbox to 0pt{\mathchardef
\nn="3236\hss$#1\nn$\kern1.4\wd0\hss}\hbox to 0pt{\hss$#1\mid$\hss}\lower.9\ht0
\hbox to 0pt{\hss$#1\smile$\hss}\kern\wd0}

\title{Positive indiscernibles}
\date{\today}
\author{Mark Kamsma}

\address[Mark Kamsma]{Department of Mathematics, Imperial College London, UK}
\email[]{mark@markkamsma.nl}
\urladdr[]{https://markkamsma.nl}

 
\begin{document}

\begin{abstract}
We generalise various theorems for finding indiscernible trees and arrays to positive logic: based on an existing modelling theorem for s-trees, we prove modelling theorems for str-trees, str$_0$-trees (the reduct of str-trees that forgets the length comparison relation) and arrays. In doing so, we prove stronger versions for basing---rather than locally basing or EM-basing---str-trees on s-trees and str$_0$-trees on str-trees. As an application we show that a thick positive theory has $k$-\TPtwo iff it has $2$-\TPtwo.
\end{abstract}

\maketitle

\tableofcontents

\section{Introduction}
As model theory in positive logic is maturing (see e.g.\ \cite{ben-yaacov_positive_2003, ben-yaacov_simplicity_2003, ben-yaacov_fondements_2007, haykazyan_existentially_2021, dobrowolski_kim-independence_2022, dmitrieva_dividing_2023}) the need for the development of tools available to us in full first-order logic becomes more and more necessary. An important notion in model-theoretic arguments is that of indiscernibles. The most popular occurrence of this is an indiscernible sequence: a sequence where any two finite subsequences have the same type. The notion of indiscernibility can be generalised by replacing the linear order that indexes an indiscernible sequence by another indexing structure. In this paper we consider various kinds of trees and an array as indexing structure. The main results are modelling theorems for positive logic, which allow us to find indiscernibles indexed by trees or arrays based on an arbitrary set of parameters indexed by the same structure, while inheriting certain local structure.

Tree indiscernibles have already found deep applications in model theory. For example, in the development of Kim-independence for NSOP$_1$ theories \cite{kaplan_kim-independence_2020, dobrowolski_kim-independence_2022}. Their original motivation stems from the proof of \cite[Theorem III.7.11]{shelah_classification_1990}, which is the celebrated theorem that a theory is simple iff it is NTP$_1$ and NTP$_2$. In that proof the existence of tree indiscernibles is claimed, and the details are filled in in \cite{kim_tree_2014}.

As any full first-order theory can be seen as a (thick) positive theory (see \thref{morleyisation}), our main results are a direct generalisation of existing results for full first-order logic \cite{kim_tree_2014, takeuchi_existence_2012, scow_indiscernibles_2015}. The main difficulty in generalising the proofs is that arguments using Ramsey's theorem break down. In these arguments one can write down a partial type that expresses indiscernibility formula by formula, using formulas of the form $\forall x_1 x_2 (\phi(x_1) \leftrightarrow \phi(x_2))$. By compactness one can then reduce to finding parameters that are indiscernible with respect to a finite set of formulas. This can be done using Ramsey's theorem by colouring using types restricted to this finite set of formulas, of which there are finitely many. In positive logic, under the assumption of  thickness, we can still write down a partial type expressing indiscernibility. However, this is no longer necessarily done formula by formula. So instead we have to use all possible types as colours, of which there are infinitely many. The solution is to replace the use of Ramsey's theorem by the Erd\H{o}s-Rado theorem. This makes arguments more complicated, but we also get stronger statements (see \thref{str-basing,str0-basing}).

\textbf{Main results.} Our main results are the following three modelling theorems. The thickness assumption is the mild assumption that being an indiscernible sequence is type-definable (\thref{thickness}). Further justification and reason for this assumption is given in the discussion after \thref{thickness,tree-tuple-extra-notation}. For the other definitions involved, we refer the reader to Section \ref{sec:generalised-indiscernibles}.
\begin{theorem}[str-modelling]
\thlabel{str-modelling}
Let $T$ be a thick theory. Let $(a_\eta)_{\eta \in \omega^{<\omega}}$ be a tree of tuples of the same length and let $C$ be any set of parameters, then there is a tree $(b_\eta)_{\eta \in \omega^{<\omega}}$ that is str-indiscernible over $C$ and $\EM_\str$-based on $(a_\eta)_{\eta \in \omega^{<\omega}}$ over $C$.
\end{theorem}
In the process of proving the above theorem we prove \thref{str-basing}, which states that given a tall enough s-indiscernible tree we can base an str-indiscernible tree on it. This statement is interesting on its own because more information is carried over when basing one tree on another, instead of EM-basing trees on each other. Similarly, in proving the following theorem we prove that str$_0$-trees can be based on str-trees, see \thref{str0-basing}.
\begin{theorem}[str$_0$-modelling]
\thlabel{str0-modelling}
Let $T$ be a thick theory. Let $(a_\eta)_{\eta \in \omega^{<\omega}}$ be a tree of tuples of the same length and let $C$ be any set of parameters, then there is a tree $(b_\eta)_{\eta \in \omega^{<\omega}}$ that is str$_0$-indiscernible over $C$ and $\EM_{\str_0}$-based on $(a_\eta)_{\eta \in \omega^{<\omega}}$ over $C$.
\end{theorem}
\begin{theorem}[array modelling]
\thlabel{array-modelling}
Let $T$ be a thick theory. Let $(a_{i,j})_{i,j < \omega}$ be an array of tuples of the same length and let $C$ be any set of parameters, then there is an array $(b_{i,j})_{i,j < \omega}$ that is array-indiscernible over $C$ and $\EM_{\ar}$-based on $(a_{i,j})_{i,j < \omega}$ over $C$.
\end{theorem}
As an application of the array modelling theorem we generalise the following from full first-order logic, completing a task from \cite[Remark 7.3]{dmitrieva_dividing_2023}.
\begin{theorem}
\thlabel{k-tp2-iff-2-tp2}
Let $T$ be a thick theory. If $\phi(x, y)$ has $k$-\TPtwo for some $k \geq 2$ then some conjunction $\bigwedge_{i = 1}^n \phi(x, y_i)$ has $2$-\TPtwo. Hence, $T$ has $k$-\TPtwo for some $k \geq 2$ iff $T$ has $2$-\TPtwo.
\end{theorem}
\textbf{Acknowledgements.} We thank the anonymous referee for their suggestions for improvement, and in particular for suggesting a simpler version of \thref{based-on-implications-non-reversible} compared to a previous version of this paper.
\section{Preliminaries}
\label{sec:preliminaries}
We only recall the definitions and facts about positive logic that we need, for a more extensive treatment and discussion see \cite{ben-yaacov_positive_2003, ben-yaacov_fondements_2007} and for a more survey-like overview see \cite[Section 2]{dobrowolski_kim-independence_2022}.
\begin{definition}
\thlabel{positive-syntax}
A \emph{positive formula} in a fixed language is one that is obtained from combining atomic formulas using $\wedge$, $\vee$, $\top$, $\bot$ and $\exists$. An \emph{h-inductive sentence} is a sentence of the form $\forall x(\phi(x) \to \psi(x))$, where $\phi(x)$ and $\psi(x)$ are positive formulas. A \emph{positive theory} is a set of h-inductive sentences.
\end{definition}
\begin{remark}
\thlabel{morleyisation}
Full first-order logic can be studied as a special case of positive logic. This is done through \emph{Morleyisation}. For this we add a relation symbol $R_\phi(x)$ to our language for every full first-order formula $\phi(x)$, and have our theory (inductively) express that $R_\phi(x)$ and $\phi(x)$ are equivalent. This way every first-order formula is equivalent to a relation symbol, and thus in particular to a positive formula.
\end{remark}
We are interested in \emph{positively closed} models. There are various characterisations (see e.g.\ \cite[Definition 2.5]{dobrowolski_kim-independence_2022}, where they are called ``existentially closed''), but we only need one.
\begin{definition}
\thlabel{existentially-closed}
We call a model $M$ of a positive theory $T$ \emph{positively closed} or \emph{p.c.}\ if for every positive formula $\phi(x)$, whenever $M \not \models \phi(a)$ then there is a positive formula $\psi(x)$ that implies $\neg \phi(x)$ modulo $T$, i.e.\ $T \models \neg \exists (\phi(x) \wedge \psi(x))$, with $M \models \psi(a)$
\end{definition}
A \emph{positive type} will be a set of positive formulas, over some parameter set $B$, satisfied by some tuple $a$ in some p.c.\ model $M$:
\[
\tp(a/B) = \{ \phi(x, b) : \phi \text{ is a positive formula, } M \models \phi(a, b) \text{ and } b \in B \}.
\]
Given a positive theory $T$, we will work in a monster model $\MM$ of that theory. For this we need to assume the \emph{joint continuation property} or JCP for $T$. This means that for any two models $M_1$ and $M_2$ of $T$ there is a third model $N$ with homomorphisms $M_1 \to N \leftarrow M_2$. This is the positive version of working in a complete theory, and we can always extend a positive theory $T$ to one with JCP by taking the set of all h-inductive sentences that are true in some p.c.\ model of $T$.

The monster model $\MM$ of $T$ can then be constructed as usual. We let the reader fix their favourite notion of smallness (e.g., fix a big enough cardinal $\kappa$, and let ``small'' mean $< \kappa$). We recall the properties of a monster model $\MM$:
\begin{itemize}
\item \emph{positively closed}, $\MM$ is a p.c.\ model of $T$;
\item \emph{very homogeneous}, for any small $a, b, C$ we have that $\tp(a/C) = \tp(b/C)$ iff there is $f \in \Aut(\MM/C)$ with $f(a) = b$ (we will also write $a \equiv_C b$);
\item \emph{very saturated}, any finitely satisfiable small set of positive formulas $\Sigma$ over $\MM$ is satisfiable in $\MM$.
\end{itemize}
As usual, we will omit the monster model from notation. For example, we write $\models \phi(a)$ instead of $\MM \models \phi(a)$.

We stress that the monster is only saturated for sets of \emph{positive} formulas, and so we can only apply compactness to such sets. This is where the main challenges in generalising arguments from the full first-order setting stem from.
\begin{definition}
\thlabel{indiscernible-sequences}
A sequence $(a_i)_{i \in I}$ (for some linear order $I$) is $C$-indiscernible if for any $i_1 < \ldots < i_n$ and $j_1 < \ldots < j_n$ in $I$ we have that $a_{i_1} \ldots a_{i_n} \equiv_C a_{j_1} \ldots a_{j_n}$.
\end{definition}
\begin{definition}
\thlabel{lascar-distance}
We write $\d_C(a, a') \leq n$ if there are $a = a_0, a_1, \ldots, a_n = a'$ such that $a_i$ and $a_{i+1}$ are on a $C$-indiscernible sequence for all $0 \leq i < n$, and we say that $a$ and $a'$ have \emph{Lascar distance} at most $n$.
\end{definition}
\begin{definition}[{\cite[Proposition 1.5]{ben-yaacov_thickness_2003}}]
\thlabel{thickness}
A positive theory $T$ is called \emph{thick} if the following equivalent conditions hold:
\begin{enumerate}[label=(\roman*)]
\item being an indiscernible sequence is type-definable, i.e.\ there is a partial type $\Theta((x_i)_{i < \omega})$ such that $\models \Theta((a_i)_{i < \omega})$ iff $(a_i)_{i < \omega}$ is an indiscernible sequence;
\item the property $\d_z(x, y) \leq n$ is type-definable for all $n < \omega$, i.e.\ for all $n < \omega$ there is a partial type $\Sigma_n(x, y, z)$ such that $\models \Sigma_n(a, a', C)$ iff $\d_C(a, a') \leq n$.
\end{enumerate}
\end{definition}
Note that it is essential in \thref{thickness} that we evaluate the partial types in a p.c.\ model, which we indeed do by working in a monster model.

Thickness is a mild assumption, as it is satisfied by many classes of positive theories.
\begin{itemize}
\item Any full first-order theory, seen as a positive theory (\thref{morleyisation}), is thick.
\item Any continuous theory in the sense of \cite{ben-yaacov_model_2008} is Hausdorff, and therefore thick.
\item Jonsson theories, and even the positive Jonsson theories from \cite{poizat_positive_2018}, by definition have the property that any span $M_1 \leftarrow M_0 \to M_2$ of homomorphisms between models can be amalgamated. This implies that these theories are Hausdorff by \cite[Theoreme 20]{ben-yaacov_fondements_2007}), and so in particular they are thick.
\item Adding hyperimaginaries, such as the $T^\text{heq}$ construction, preserves thickness, see \cite[Theorem 10.17]{dobrowolski_kim-independence_2022}.
\item Specific examples include the positive NIP theories in \cite[Section 6.3]{dobrowolski_amalgamation_2023} and the positive theory of exponential fields from \cite{haykazyan_existentially_2021}. These are both in the class of positive Jonsson theories. Another specific example, which is not in any of the classes mentioned before, is the positive theory of bilinear spaces over a fixed infinite field from \cite{kamsma_bilinear_2023}, which is semi-Hausdorff and thus thick (see Proposition 4.14 there).
\end{itemize}
Furthermore, neostability theory, where results like the ones in this paper often find their applications, works best under the thickness assumption. For example, without thickness a stable theory may not be simple in the sense that local character fails for dividing \cite[Example 4.3]{ben-yaacov_simplicity_2003}. In \cite{dobrowolski_kim-independence_2022} thickness is crucial for the existence of Lascar-invariant types, which are the basis for Kim-dividing in positive NSOP$_1$ theories.\\
\\
\noindent
\textbf{Conventions.}
\begin{itemize}
\item Whenever we say ``formula'', ``type'' or ``theory'' we will mean ``positive formula'', ``positive type'' and ``positive theory'' respectively, unless explicitly stated otherwise. This also means that every formula, type and theory we consider will be implicitly assumed to be positive.
\item Let $I$ be some indexing set and suppose that we have an $I$-indexed set of variables $(x_i)_{i \in I}$ or parameters $(a_i)_{i \in I}$. Then for any tuple $\bar{\eta} = (\eta_1, \ldots, \eta_n)$ in $I$ we write $x_{\bar{\eta}}$ and $a_{\bar{\eta}}$ for the tuples $(x_{\eta_1}, \ldots, x_{\eta_n})$ and $(a_{\eta_1}, \ldots, a_{\eta_n})$ respectively.
\end{itemize}
\section{Generalised indiscernibles}
\label{sec:generalised-indiscernibles}
The following idea stems from \cite[Definition VII.2.4]{shelah_classification_1990}.
\begin{definition}
\thlabel{generalised-indiscernibles}
Let $\L$ be a language (which we always assume to include the equality symbol), $I$ an $\L$-structure, and $(a_i)_{i \in I}$ an $I$-indexed set of parameters. We will refer to $I$ as the \emph{indexing structure}. Let $C$ be any parameter set. We say that $(a_i)_{i \in I}$ is \emph{$I$-indiscernible over $C$} if for any two tuples $\bar{\eta}$ and $\bar{\nu}$ in $I$ we have that
\[
\qftp_\L(\bar{\eta}) = \qftp_\L(\bar{\nu}) \implies a_{\bar{\eta}} \equiv_C a_{\bar{\nu}},
\]
where $\qftp_\L(\bar{\eta})$ is the quantifier-free $\L$-type of $\bar{\eta}$.
\end{definition}
At the end of Section \ref{sec:preliminaries} we made a convention about every formula and type being positive. For the quantifier-free $\L$-types in the indexing structures this has no real effect. This is because we will only be interested in whether or not such quantifier-free $\L$-types are equal, and two tuples have the same quantifier-free $\L$-type in the full first-order sense if and only if they satisfy the same atomic $\L$-formulas. In other words, for the quantifier-free $\L$-formulas, where $\L$ is the language of some indexing structure, we may as well allow the $\neg$ symbol.
\begin{example}
\thlabel{indiscernible-sequence-as-generalised-indiscernible}
We let $\L_< = \{ < \}$ be the language with a single ordering symbol and consider $\omega$ as an $\L_<$-structure with the usual ordering. Then $(a_i)_{i < \omega}$ being $\omega$-indiscernible over $C$ means precisely that $(a_i)_{i < \omega}$ is a $C$-indiscernible sequence.
\end{example}
\begin{definition}
\thlabel{based-on}
Let $\L$ be some language and let $(a_i)_{i \in I}$ and $(b_i)_{i \in I'}$ be two sets of parameters indexed by $\L$-structures $I$ and $I'$ respectively. Let furthermore $C$ be any parameter set. We say that \emph{$(b_i)_{i \in I'}$ is $\L$-based on $(a_i)_{i \in I}$ over $C$} if for any finite tuple $\bar{\eta}$ in $I'$ there is a tuple $\bar{\nu}$ in $I$ such that $\qftp_\L(\bar{\nu}) = \qftp_\L(\bar{\eta})$ and $b_{\bar{\eta}} \equiv_C a_{\bar{\nu}}$.

We say that \emph{$(b_i)_{i \in I'}$ is locally $\L$-based on $(a_i)_{i \in I}$ over $C$} if for any finite tuple $\bar{\eta}$ in $I'$ and any formula $\phi(x_{\bar{\eta}})$ over $C$ such that $\models \phi(b_{\bar{\eta}})$ there is a tuple $\bar{\nu}$ in $I$ such that $\qftp_\L(\bar{\nu}) = \qftp_\L(\bar{\eta})$ and $\models \phi(a_{\bar{\nu}})$.
\end{definition}
We note the difference in terminology from \cite[Definition 3.8]{kim_tree_2014}, where ``based on'' is used for what we call ``locally based on'', a distinction that is further promoted in \cite[Definition 2.5]{scow_indiscernibles_2015}. The difference is important, see \thref{based-on-implications-non-reversible}. Another difference is that we call it ``(locally) $\L$-based on'' instead of ``(locally) $I$-based on'', this is because we wish to compare sets of parameters indexed by different structures in the same language.
\begin{remark}
\thlabel{locally-based-on-alternative-formulation}
An alternative formulation to $(b_i)_{i \in I'}$ being locally $\L$-based on $(a_i)_{i \in I}$ over $C$ is the following (see e.g. \cite[Definition 2.5]{scow_indiscernibles_2015}). For any finite set of formulas $\Phi$ with parameters in $C$ denote by $\tp_\Phi(a/C)$ the restriction of $\tp(a/C)$ to the formulas in $\Phi$. Then $(b_i)_{i \in I'}$ is locally $\L$-based on $(a_i)_{i \in I}$ over $C$ if for every finite set $\Phi$ of formulas over $C$ and any finite tuple $\bar{\eta}$ in $I'$ there is a finite tuple $\bar{\nu}$ in $I$ such that $\tp_\Phi(b_{\bar{\eta}}/C) = \tp_\Phi(a_{\bar{\nu}}/C)$. This is indeed an equivalent formulation, even in positive logic. It clearly implies the formulation in \thref{based-on}, so we prove the converse.

Let $\phi \in \Phi$ be some formula and let $\bar{\eta}' \subseteq \bar{\eta}$ be a tuple whose length matches the number of free variables in $\phi$. Define $\psi_{\phi, \bar{\eta}'}(x_{\bar{\eta}'})$ as follows: if $\models \phi(b_{\bar{\eta}'})$ then take $\psi_{\phi, \bar{\eta}'}(x_{\bar{\eta}'})$ to be $\phi(x_{\bar{\eta}'})$, otherwise there is $\chi((x_i)_{i \in \bar{\eta}'})$ that implies $\neg \phi(x_{\bar{\eta}'})$ modulo $T$ such that $\models \chi(b_{\bar{\eta}'})$ and we take $\psi_{\phi, \bar{\eta}'}(x_{\bar{\eta}'})$ to be $\chi(x_{\bar{\eta}'})$. Now let
\[
\psi(x_{\bar{\eta}}) = \bigwedge \{ \psi_{\phi, \bar{\eta}'}(x_{\bar{\eta}'}) : \phi \in \Phi \text{ and } \bar{\eta}' \subseteq \bar{\eta} \text{ matching the free variables in } \phi \}.
\]
By construction $\models \psi(b_{\bar{\eta}})$ and so there is $\bar{\nu}$ in $I$ such that $\models \psi(a_{\bar{\nu}})$, from which it follows that $\tp_\Phi(b_{\bar{\eta}}/C) = \tp_\Phi(a_{\bar{\nu}}/C)$.
\end{remark}
\begin{example}
\thlabel{indiscernible-sequence-basing-on}
We recall the following fact for finding indiscernible sequences \cite[Lemma 1.2]{ben-yaacov_simplicity_2003}. Let $C$ be any parameter set and let $\kappa$ be any cardinal, and set $\lambda = \beth_{(2^{|T|+|C|+\kappa})^+}$. Then for any sequence $(a_i)_{i < \lambda}$ of $\kappa$-tuples there is a $C$-indiscernible sequence $(b_i)_{i < \omega}$ such that for all $n < \omega$ there are $i_1 < \ldots < i_n < \lambda$ with $b_1 \ldots b_n \equiv_C a_{i_1} \ldots a_{i_n}$.

With $C$, $\kappa$ and $\lambda$ as above, this is exactly saying that for any sequence $(a_i)_{i < \lambda}$ there is $C$-indiscernible $(b_i)_{i < \omega}$ that is $\L_<$-based on $(a_i)_{i < \lambda}$ over $C$.
\end{example}
\begin{definition}
\thlabel{em-type}
Let $I$ be a structure in some language $\L$, let $(a_i)_{i \in I}$ be an $I$-indexed set of parameters and let $C$ be some set of parameters. We write
\begin{align*}
\EM_\L((a_i)_{i \in I} / C) = \{ \phi(x_{\bar{\eta}}) \ :\ &\phi(x_{\bar{\eta}}) \text{ is a formula with parameters in } C \text{ and }\\
&\models \phi(a_{\bar{\nu}}) \text{ for all } \bar{\nu} \text{ with } \qftp_\L(\bar{\nu}) = \qftp_\L(\bar{\eta}) \}.
\end{align*}
for the \emph{$\EM_\L$-type of $(a_i)_{i \in I}$ over $C$}, which is a partial type in variables $(x_i)_{i \in I}$. For $\phi(x_{\bar{\eta}}) \in \EM_\L((a_i)_{i \in I}/C)$ we call $\qftp_\L(\bar{\eta})$ the \emph{associated quantifier-free $\L$-type}.

Let $I'$ be a second $\L$-structure. We say that \emph{$(b_i)_{i \in I'}$ is $\EM_\L$-based on $(a_i)_{i \in I}$ over $C$} if the following holds: for any $\phi \in \EM_\L((a_i)_{i \in I}/C)$ with associated quantifier-free $\L$-type $\Delta$, and any tuple $\bar{\eta}$ in $I'$ realising $\Delta$ we have that $\models \phi(b_{\bar{\eta}})$.
\end{definition}
The name ``EM-type'' is short for ``Ehrenfeucht-Mostowski type'' and the above is a straightforward generalisation from the traditional case where the indexing structure is a linear order (see e.g.\ \cite[Definition 5.1.2]{tent_course_2012}).

It could be the case that $(b_i)_{i \in I'}$ is (locally) $\L$-based or $\EM_\L$-based on $(a_i)_{i \in I}$ because $I$ does not realise a quantifier-free $\L$-type that is realised in $I'$, or vice versa. For example, for any sequence $(a_i)_{i < \omega}$ we have that $a_0$ (as a singleton indexed set) is $\L_<$-based on $(a_i)_{i < \omega}$. To circumvent this issue we introduce the following definition, which will be satisfied by all our indexing structures of interest.
\begin{definition}
\thlabel{qftp-comparable}
Two $\L$-structures $I$ and $I'$ are \emph{qftp-comparable} if $I'$ realises every quantifier-free $\L$-type in finitely many variables that is realised in $I$, and vice versa.
\end{definition}
\begin{remark}
\thlabel{basing-transitive}
Let $I$, $I'$ and $I''$ be pairwise qftp-comparable $\L$-structures. Suppose that $(d_i)_{i \in I''}$ is $\L$-based on $(b_i)_{i \in I'}$ over some parameter set $C$ and $(b_i)_{i \in I'}$ is $\L$-based on $(a_i)_{i \in I}$ over $C$. Then $(d_i)_{i \in I''}$ is $\L$-based on $(a_i)_{i \in I}$ over $C$. We call this ``transitivity of $\L$-basing'', and similarly for locally $\L$-basing and $\EM_\L$-basing.
\end{remark}
\begin{remark}
\thlabel{em-basing-alternative}
Let $I$ and $I'$ be two qftp-comparable $\L$-structures. Then we can equivalently be phrase the definition of being $\EM_\L$-based as follows. Define
\begin{align*}
\Sigma((x_i)_{i \in I'}) = \{ \phi(x_{\bar{\eta}}) \ :\ &\bar{\eta} \text{ is a finite tuple in } I' \text{ and}\\
&\text{for some tuple } \bar{\nu} \text{ in } I \text{ with } \qftp_\L(\bar{\nu}) = \qftp_\L(\bar{\eta})\\
&\text{we have that } \phi(x_{\bar{\nu}}) \in \EM_\L((a_i)_{i \in I}/C) \}.
\end{align*}
Then $(b_i)_{i \in I'}$ is $\EM_\L$-based on $(a_i)_{i \in I}$ over $C$ iff $\Sigma((x_i)_{i \in I'}) \subseteq \EM_\L((b_i)_{i \in I'}/C)$ iff $\models \Sigma((b_i)_{i \in I'})$. Note that when $I = I'$ then $\Sigma((x_i)_{i \in I}) = \EM_\L((a_i)_{i \in I}/C)$, so in that case we get that $(b_i)_{i \in I}$ is $\EM_\L$-based on $(a_i)_{i \in I}$ over $C$ iff $\EM_\L((a_i)_{i \in I}/C) \subseteq \EM_\L((b_i)_{i \in I}/C)$ iff $(b_i)_{i \in I} \models \EM_\L((a_i)_{i \in I}/C)$.
\end{remark}
\begin{proposition}
\thlabel{based-on-implies-locally-based-on-implies-em-based-on}
Let $(a_i)_{i \in I}$ and $(b_i)_{i \in I'}$ be sets of parameters indexed by qftp-comparable $\L$-structures $I$ and $I'$ respectively. For any parameter set $C$ we have that (i) $\Rightarrow$ (ii) $\Rightarrow$ (iii) as below:
\begin{enumerate}[label=(\roman*)]
\item $(b_i)_{i \in I'}$ is $\L$-based on $(a_i)_{i \in I}$ over $C$,
\item $(b_i)_{i \in I'}$ is locally $\L$-based on $(a_i)_{i \in I}$ over $C$,
\item $(b_i)_{i \in I'}$ is $\EM_\L$-based on $(a_i)_{i \in I}$ over $C$.
\end{enumerate}
\end{proposition}
\begin{proof}
The implication (i) $\implies$ (ii) is clear (even without the assumption of being qftp-comparable), we prove (ii) $\implies$ (iii). So let $\phi \in \EM_\L((a_i)_{i \in I}/C)$ and let $\Delta$ be the associated quantifier-free $\L$-type. Let $\bar{\eta}$ be a tuple in $I'$ satisfying $\Delta$, which exists by qftp-comparability. Suppose for a contradiction that $\not \models \phi(b_{\bar{\eta}})$. Then there is $\psi(\bar{x})$ that implies $\neg \phi(\bar{x})$ modulo $T$ such that $\models \psi(b_{\bar{\eta}})$. As $(b_i)_{i \in I'}$ is locally $\L$-based on $(a_i)_{i \in I}$ over $C$ there must then be $\bar{\nu}$ with $\qftp_\L(\bar{\nu}) = \qftp_\L(\bar{\eta})$ such that $\models \psi(a_{\bar{\nu}})$. However, that means that $\not \models \phi(a_{\bar{\nu}})$, contradicting $\phi \in \EM_\L((a_i)_{i \in I}/C)$.
\end{proof}
\begin{example}
\thlabel{based-on-implications-non-reversible}
None of the implications in \thref{based-on-implies-locally-based-on-implies-em-based-on} are reversible, not even assuming thickness. As an example we will consider a language with countably many constants $(c_i)_{i < \omega}$ and the structure $M$ consisting of only the constants, which are all interpreted as distinct elements.

First we consider the full first-order theory of $M$ and work in a monster model of this theory (in the full first-order sense). This theory has quantifier elimination, because every type is determined by its quantifier-free part. Let $(a_i)_{i < \omega}$ be a sequence of distinct elements in the monster that are not equal to any of the constant symbols. Then $(a_i)_{i < \omega}$ is locally $\L_<$-based on the sequence $(c_i)_{i < \omega}$ that enumerates the constants, but it is not $\L_<$-based on this sequence (in both cases over $\emptyset$).

Next we consider the positive theory $T$ of $M$ (i.e.\ all h-inductive sentences true in $M$). So $T$ just expresses that all the constants are distinct, and $M$ is the only p.c.\ model of $T$ (up to isomorphism) and is thus the monster. Consider the sequence $(b_i)_{i < \omega}$ with constant value $c_0$, so $b_i = c_0$ for all $i < \omega$. Then $(b_i)_{i < \omega}$ is $\EM_{\L_<}$-based on $(c_i)_{i < \omega}$ but it is not locally $\L_<$-based on $(c_i)_{i < \omega}$ (in both cases over $\emptyset$).

Note that in the last case we really needed to consider an example in positive logic. The key is that $x_0 \neq x_1$ is not a positive formula and is thus not in the $\EM_{\L_<}$-type. In fact, as is well known, in full first-order logic we do have that (ii) and (iii) from \thref{based-on-implies-locally-based-on-implies-em-based-on} are equivalent (for qftp-comparable structures). We include a proof for completeness' sake, and to point out the usage of negation.

Let $\bar{\eta}$ be any finite tuple in $I'$ and let $\phi(x_{\bar{\eta}})$ be any formula over $C$ such that $\models \phi(b_{\bar{\eta}})$. Suppose for a contradiction that there is no tuple $\bar{\nu}$ in $I$ with $\qftp_\L(\bar{\nu}) = \qftp_\L(\bar{\eta})$ such that $\models \phi(a_{\bar{\nu}})$. Then $\models \neg \phi(a_{\bar{\nu}})$ for all such $\bar{\nu}$. Now picking one such $\bar{\nu}$, which exists by qftp-comparability, we get $\neg \phi(x_{\bar{\nu}}) \in \EM_\L((a_i)_{i \in I}/C)$. This is a contradiction, because then $\models \neg \phi(b_{\bar{\eta}})$ as $(b_i)_{i \in I'}$ is $\EM_\L$-based on $(a_i)_{i \in I}$ over $C$.
\end{example}
\begin{remark}
\thlabel{equivalent-of-em-dual-to-locally-based}
Versions of the proof of \thref{based-on-implies-locally-based-on-implies-em-based-on} and the argument at the end of \thref{based-on-implications-non-reversible} reveal an equivalent formulation of being $\EM_\L$-based, which is a negative version of being locally $\L$-based. Namely, $(b_i)_{i \in I'}$ is $\EM_\L$-based on $(a_i)_{i \in I}$ over $C$ if for any finite tuple $\bar{\eta}$ in $I'$ and any formula $\phi(x_{\bar{\eta}})$ over $C$ such that $\not \models \phi(b_{\bar{\eta}})$ there is a tuple $\bar{\nu}$ in $I$ such that $\qftp_\L(\bar{\nu}) = \qftp_\L(\bar{\eta})$ and $\not \models \phi(a_{\bar{\nu}})$.

We will continue working with the $\EM$-type perspective, because that makes it immediately clear that type-definable behaviour is captured and thus carried over.
\end{remark}
\begin{remark}
\thlabel{sequence-em-modelling}
In \thref{indiscernible-sequence-basing-on} we saw that we can base an indiscernible sequence on a sufficiently long sequence. In \thref{based-on-implications-non-reversible} we saw examples of indiscernible sequences that are $\EM_{\L_<}$-based on sequences of length $\omega$ (the sequences $(a_i)_{i < \omega}$ and $(b_i)_{i < \omega}$ there are indiscernible). This can always be done and illustrates the use of $\EM$-types.

That is, for any sequence $(a_i)_{i < \omega}$ of tuples of the same length and any parameter set $C$ there is a $C$-indiscernible sequence $(b_i)_{i < \omega}$ that is $\EM_{\L_<}$-based on $(a_i)_{i < \omega}$ over $C$. To see this we let $\lambda$ be the cardinal from \thref{indiscernible-sequence-basing-on}. Then we define
\begin{align*}
\Sigma((x_i)_{i < \lambda}) = \{ \phi(x_{i_1}, \ldots, x_{i_n}) \ :\ &\phi(x_1, \ldots, x_n) \in \EM_{\L_<}((a_i)_{i < \omega}/C)\\
&\text{and }i_1 < \ldots < i_n < \lambda \},
\end{align*}
note that this is the same construction of $\Sigma$ as in \thref{em-basing-alternative}, just specialised to $\L_<$. Let $(a'_i)_{i < \lambda}$ be any realisation of $\Sigma$, which exists by compactness because every finite part of $\Sigma$ is realised by $(a_i)_{i < \omega}$. Then use \thref{indiscernible-sequence-basing-on} to $\L_<$-base a $C$-indiscernible sequence $(b_i)_{i < \omega}$ on $(a'_i)_{i < \lambda}$. By \thref{based-on-implies-locally-based-on-implies-em-based-on} this means that $(b_i)_{i < \omega}$ is in particular $\EM_{\L_<}$-based on $(a'_i)_{i < \lambda}$ over $C$, which is in turn $\EM_{\L_<}$-based on $(a_i)_{i < \omega}$ over $C$ by construction. Since $\EM$-basing is transitive, we conclude that $(b_i)_{i < \omega}$ is as required.
\end{remark}
In the remainder of this section we provide tools to deal with carrying over indiscernibility and $\EM$-types between indexing structures in different languages. We refer to \thref{trees-examples} for examples.
\begin{definition}[{\cite[Definition 3.2]{scow_ramsey_2021}}]
\thlabel{qftp-respecting}
Let $\L$ and $\L'$ be languages and $I$ and $I'$ be structures in those respective languages. We call a function $f: I \to I'$ \emph{qftp-respecting} if for any two finite tuples $\bar{\eta}$ and $\bar{\nu}$ in $I$ we have that $\qftp_\L(\bar{\eta}) = \qftp_\L(\bar{\nu})$ implies $\qftp_{\L'}(f(\bar{\eta})) = \qftp_{\L'}(f(\bar{\nu}))$.
\end{definition}
\begin{lemma}[Re-indexing lemma]
\thlabel{reindexing-lemma}
Let $f: I \to I'$ be a qftp-respecting function between structures $I$ and $I'$ in languages $\L$ and $\L'$ respectively. Let $C$ be any parameter set.
\begin{enumerate}[label=(\roman*)]
\item Let $(a'_i)_{i \in I'}$ and $(a_i)_{i \in I}$ be such that $a_i = a'_{f(i)}$ for all $i \in I$. If $(a'_i)_{i \in I'}$ is $I'$-indiscernible over $C$ then $(a_i)_{i \in I}$ is $I$-indiscernible over $C$.
\item Let $(a'_i)_{i \in I'}$, $(b'_i)_{i \in I'}$, $(a_i)_{i \in I}$ and $(b_i)_{i \in I}$ be such that $a_i = a'_{f(i)}$ and $b_i = b'_{f(i)}$ for all $i \in I$. If $f$ is surjective and $(b_i)_{i \in I}$ is $\EM_\L$-based on $(a_i)_{i \in I}$ over $C$ then $(b'_i)_{i \in I'}$ is $\EM_{\L'}$-based on $(a'_i)_{i \in I'}$ over $C$.
\item Suppose that there is $g: J \to I$, where $J$ is an $\L'$-structure, such that $fg$ is an $\L'$-embedding. Let $(a'_i)_{i \in I'}$, $(b_i)_{i \in I}$, $(a_i)_{i \in I}$ and $(b'_j)_{j \in J}$ be such that $a_i = a'_{f(i)}$ and $b'_{j} = b_{g(j)}$ for all $i \in I$ and $j \in J$. If $(b_i)_{i \in I}$ is $\EM_\L$-based on $(a_i)_{i \in I}$ over $C$ then $(b'_j)_{j \in J}$ is $\EM_{\L'}$-based on $(a'_i)_{i \in I'}$ over $C$.
\end{enumerate}
\end{lemma}
\begin{proof}
~
\begin{enumerate}[label=(\roman*)]
\item Let $\bar{\eta}$ and $\bar{\nu}$ be tuples in $I$ such that $\qftp_\L(\bar{\eta}) = \qftp_\L(\bar{\nu})$ then $\qftp_{\L'}(f(\bar{\eta})) = \qftp_{\L'}(f(\bar{\nu}))$ because $f$ is qftp-respecting. Hence, by $I'$-indiscernibility of $(a'_i)_{i \in I'}$ we have that $a_{\bar{\eta}} = a'_{f(\bar{\eta})} \equiv_C a'_{f(\bar{\nu})} = a_{\bar{\nu}}$.
\item Let $\phi \in \EM_{\L'}((a'_i)_{i \in I'}/C)$ and let $\Delta'$ be the associated quantifier-free $\L'$-type. Let $\bar{\eta}'$ be any tuple in $I'$ satisfying $\Delta'$. By surjectivity there is a tuple $\bar{\eta}$ in $I$ such that $f(\bar{\eta}) = \bar{\eta}'$. For any tuple $\bar{\nu}$ in $I$ such that $\qftp_\L(\bar{\nu}) = \qftp_\L(\bar{\eta})$ we have that $\qftp_{\L'}(f(\bar{\nu})) = \qftp_{\L'}(f(\bar{\eta})) = \Delta'$ because $f$ is qftp-respecting. So $\models \phi(a'_{f(\bar{\nu})})$, and hence $\models \phi(a_{\bar{\nu}})$. As $\bar{\nu}$ as arbitrary satisfying $\qftp_\L(\bar{\eta})$, we see that $\phi \in \EM_\L((a_i)_{i \in I}/C)$ with $\qftp_\L(\bar{\eta})$ as the associated quantifier-free $\L$-type. Since $(b_i)_{i \in I}$ is $\EM_\L$-based (over $C$) on $(a_i)_{i \in I}$, we get that $\models \phi(b_{\bar{\eta}})$ and hence $\models \phi(b'_{\bar{\eta}'})$. As $\bar{\eta}'$ was arbitrary satisfying $\Delta'$ we conclude that $\phi \in \EM_{\L'}((b'_i)_{i \in I'}/C)$, as required.
\item Let $\phi \in \EM_{\L'}((a'_i)_{i \in I'}/C)$ and let $\Delta'$ be the associated quantifier-free $\L'$-type. Let $\bar{\eta}'$ be any tuple in $J$ satisfying $\Delta'$. Let $\bar{\nu}$ be any tuple in $I$ satisfying $\qftp_\L(g(\bar{\eta}'))$. Then $\qftp_{\L'}(f(\bar{\nu})) = \qftp_{\L'}(fg(\bar{\eta}')) = \Delta'$ because $f$ is qftp-respecting and because of our assumption on $fg$. We thus have that $\models \phi(a'_{f(\bar{\nu})})$ and so $\models \phi(a_{\bar{\nu}})$. As $\bar{\nu}$ was arbitrary satisfying $\qftp_\L(g(\bar{\eta}'))$ and such $\bar{\nu}$ exists (take $\bar{\nu} = g(\bar{\eta}')$) we get that $\phi \in \EM_\L((a_i)_{i \in I}/C)$ with $\qftp_\L(g(\bar{\eta}'))$ as the associated quantifier-free $\L$-type. Since $(b_i)_{i \in I}$ is $\EM_\L$-based on $(a_i)_{i \in I}$ over $C$, we get that $\models \phi(b_{g(\bar{\eta}')})$ and hence $\models \phi(b'_{\bar{\eta}'})$. As $\bar{\eta}'$ was arbitrary satisfying $\Delta'$ we conclude that $\phi \in \EM_{\L'}((b'_j)_{j \in J}/C)$, as required.
\end{enumerate}
\end{proof}
\section{The tree modelling theorems}
\begin{definition}[{\cite[Definition 2.1]{kim_tree_2014}}]
\thlabel{tree-language}
For any ordinals $\alpha$ and $\beta$ we view the set $\alpha^{<\beta}$ of functions $\eta: \gamma \to \alpha$ with $\gamma < \beta$ as a tree with the usual partial ordering: $\eta \unlhd \nu$ iff $\nu$ is an extension of $\eta$. We will simultaneously view a function $\eta: \gamma \to \alpha$ as a sequence of elements in $\alpha$ of length $\gamma$. We put further structure on $\alpha^{<\beta}$ as follows, where $\eta, \nu \in \alpha^{<\beta}$:
\begin{itemize}
\item we write $\eta \wedge \nu$ for the \emph{meet} of $\eta$ and $\nu$ with respect to $\unlhd$, i.e.\ the largest initial segment shared by $\eta$ and $\nu$;
\item we write $\eta <_\lex \nu$ for the \emph{lexicographical ordering}, i.e.\ either $\eta \lhd \nu$ or $\eta$ and $\nu$ are incomparable with respect to $\unlhd$ and for the least ordinal $\gamma$ such that $\eta(\gamma) \neq \nu(\gamma)$ we have that $\eta(\gamma) < \nu(\gamma)$;
\item we write $\ell(\eta)$ for the \emph{length} or \emph{level} of $\eta$, i.e.\ $\ell(\eta)$ is the domain of $\eta$;
\item we write $\eta <_\len \nu$ iff $\ell(\eta) < \ell(\nu)$;
\item for $\gamma < \beta$ we let $P_\gamma = \{ \eta \in \alpha^{<\beta} : \ell(\eta) = \gamma \}$.
\end{itemize}
Based on the above we put different structures on $\alpha^{<\beta}$ using the following languages.
\begin{itemize}
\item the \emph{Shelah language} $\L_\s = \{ \unlhd, \wedge, <_\lex, (P_\gamma)_{\gamma < \beta} \}$,
\item the \emph{strong Shelah language} $\L_\str = \{ \unlhd, \wedge, <_\lex, <_\len \}$,
\item the language $\L_{\str_0} = \{ \unlhd, \wedge, <_\lex \}$.
\end{itemize}
To abbreviate notation we will write ``$\EM_\s$'' and ``s-based'' for ``$\EM_{\L_\s}$'' and ``$\L_\s$-based''. Whenever we consider a tree $\alpha^{<\beta}$ as an $\L_\s$-structure we will call it an \emph{s-tree} and write ``s-indiscernible'' instead of ``$\alpha^{<\beta}$-indiscernible''. For any two tuples $\bar{\eta}$ and $\bar{\nu}$ in $\alpha^{<\beta}$ we will write $\bar{\eta} \equivsqf \bar{\nu}$ for $\qftp_{\L_\s}(\bar{\eta}) = \qftp_{\L_\s}(\bar{\nu})$.

We abbreviate notation involving $\L_\str$ and $\L_{\str_0}$ in a similar way, replacing the ``s'' in the above by ``str'' or ``str$_0$'' respectively.
\end{definition}
\begin{definition}
\thlabel{notation-for-tree-nodes}
For $\gamma_0, \ldots, \gamma_{n-1} \in \alpha$ we write $\langle \gamma_0, \ldots, \gamma_{n-1} \rangle$ for the function $i \mapsto \gamma_i$, which is an element of $\alpha^{<\beta}$. For $\eta, \nu \in \alpha^{<\beta}$ we write $\eta^\frown \nu$ for the concatenation of $\eta$ and $\nu$. Formally:
\[
(\eta^\frown \nu)(i) = \begin{cases}
\eta(i) & \text{if } i < \ell(\eta) \\
\nu(j) & \text{if } i = \ell(\eta) + j
\end{cases}
\]
For $n < \omega$ we will write $\eta^n$ for the concatenation of $\eta$ with itself $n$ times:
\[
\eta^n = \underbrace{\eta^\frown \eta^\frown \ldots {}^\frown \eta}_{n \text{ times}}
\]
So $\langle 0 \rangle^n$ is the sequence of $n$ zeroes, or formally: the constant function $n \to \{0\}$.
\end{definition}
\begin{example}
\thlabel{trees-examples}
We apply the re-indexing lemma (\thref{reindexing-lemma}) to the different tree structures.
\begin{enumerate}[label=(\roman*)]
\item Consider the identity function $f: \omega^{<\omega} \to \omega^{<\omega}$, where the domain carries the structure of an str-tree and the codomain that of an str$_0$-tree. Then $f$ is just a reduct of structures, and hence qftp-respecting. By \thref{reindexing-lemma}(i) we thus see that if a tree $(a_\eta)_{\eta \in \omega^{<\omega}}$ is str$_0$-indiscernible over $C$ then it is also str-indiscernible over $C$. Similarly, we could consider the domain and codomain of $f$ to be an s-tree and str-tree respectively instead. Now $f$ is not simply a reduct, but it is still qftp-respecting as the predicates $(P_n)_{n < \omega}$ determine the relation $<_\len$. So if $(a_\eta)_{\eta \in \omega^{<\omega}}$ is str-indiscernible over $C$ then it is s-indiscernible over $C$.
\item Like in the previous point, we consider the identity function $f: \omega^{<\omega} \to \omega^{<\omega}$ as a qftp-respecting function between an str-structure and str$_0$-structure, or between an s-structure and an str-structure. We get the following from \thref{reindexing-lemma}(ii): if $(b_\eta)_{\eta \in \omega^{<\omega}}$ is $\EM_\s$-based on $(a_\eta)_{\eta \in \omega^{<\omega}}$ over $C$ then it is also $\EM_\str$-based on $(a_\eta)_{\eta \in \omega^{<\omega}}$ over $C$, and if $(b_\eta)_{\eta \in \omega^{<\omega}}$ is $\EM_\str$-based on $(a_\eta)_{\eta \in \omega^{<\omega}}$ over $C$ then it is also $\EM_{\str_0}$-based on $(a_\eta)_{\eta \in \omega^{<\omega}}$ over $C$.
\item Let $L \subseteq \omega$ be an infinite set, and enumerate $L$ in order as $\ell_0 < \ell_1 < \ldots$. For $\eta \in \omega^{<\omega}$ of length $n$ we define $\nu_\eta \in \omega^{<\omega}$ of length $\ell_n$ as:
\[
\nu_\eta(k) = \begin{cases}
\eta(i) & \text{if } k = \ell_{i+1} - 1 \\
0 & \text{otherwise}
\end{cases}
\]
Define $f_L: \omega^{<\omega} \to \omega^{<\omega}$ as $f_L(\eta) = \nu_\eta$. The picture to keep in mind here is that of including one tree into another by only having the levels in $L$ taking non-zero values. If we consider the domain and codomain of $f_L$ as s-trees then $f_L$ is qftp-respecting. We write $\omega^{<\omega} {\upharpoonright_L}$ for the image of $f_L$ with the induced s-structure, and call this \emph{the restriction of $\omega^{<\omega}$ to levels $L$}. If $(a_\eta)_{\eta \in \omega^{<\omega}}$ is s-indiscernible over $C$ then so is $(a_\eta)_{\eta \in \omega^{<\omega} {\upharpoonright_L}}$. This construction works in the exact same way for str-trees and str$_0$-trees.
\item Fix some $\eta \in \omega^{<\omega}$ and consider $f_\eta: \omega \to \omega^{<\omega}, i \mapsto \eta^\frown \langle i \rangle$. So $f$ injects the linear order $\omega$ into the tree by sending it to the immediate successors of $\eta$. We consider $\omega$ as the usual $\L_<$-structure and $\omega^{<\omega}$ as an s-structure. Then $f_\eta$ is qftp-respecting. So if $(a_\eta)_{\eta \in \omega^{<\omega}}$ is s-indiscernible over $C$ then $(a'_i)_{i < \omega}$ defined by $a'_i = a_{f_\eta(i)}$ is a $C$-indiscernible sequence, by \thref{reindexing-lemma}(i).
\end{enumerate}
\end{example}
The following notation will be useful in proofs. We define it for finite tuples, but it would make sense for infinite tuples.
\begin{definition}
\thlabel{tree-tuple-extra-notation}
Let $\bar{\eta} = (\eta_1, \ldots, \eta_n)$ be a tuple in a tree $\alpha^{<\beta}$. We define:
\begin{itemize}
\item $\ell(\bar{\eta}) = \{ \ell(\eta_i) : 1 \leq i \leq n \}$ to be the set of levels of the elements in $\bar{\eta}$,
\item $\cl_\wedge(\bar{\eta})$ to be the closure of $\bar{\eta}$ under meets.
\end{itemize}
\end{definition}
The thickness assumption in the main theorems is needed to make the various forms of indiscernibility type-definable, which we make precise below in \thref{s-str-str0-indiscernible-type-definable}. We also rely on a result from \cite{dobrowolski_kim-independence_2022} (see \thref{s-modelling}), whose proof heavily relies on s-indiscernibility being type-definable. While we do not exclude the possibility of achieving the same results without the thickness assumption, it seems unlikely. Meanwhile, the discussion after \thref{thickness} shows the relevance of our results, even with the thickness assumption.
\begin{proposition}
\thlabel{s-str-str0-indiscernible-type-definable}
Let $T$ be a thick theory. The properties of being s-indiscernible, being str-indiscernible and being str$_0$-indiscernible are type-definable. This is done by taking the partial type that states that any two tuples of variables, whose tuples of indices have the same quantifier-free type (in the relevant language), have Lascar distance at most $2$.

More precisely, define the partial type $\pi_\s((x_\eta)_{\eta \in \omega^{<\omega}}, y)$ to be
\[
\bigcup\{ \d_y(x_{\bar{\eta}}, x_{\bar{\nu}}) \leq 2 : \bar{\eta}, \bar{\nu} \text{ are finite tuples in } \omega^{<\omega} \text{ with } \bar{\eta} \equivsqf \bar{\nu} \},
\]
define $\pi_\str((x_\eta)_{\eta \in \omega^{<\omega}}, y)$ to be
\[
\bigcup\{ \d_y(x_{\bar{\eta}}, x_{\bar{\nu}}) \leq 2 : \bar{\eta}, \bar{\nu} \text{ are finite tuples in } \omega^{<\omega} \text{ with } \bar{\eta} \equivstrqf \bar{\nu} \},
\]
and define $\pi_{\str_0}((x_\eta)_{\eta \in \omega^{<\omega}}, y)$ to be
\[
\bigcup\{ \d_y(x_{\bar{\eta}}, x_{\bar{\nu}}) \leq 2 : \bar{\eta}, \bar{\nu} \text{ are finite tuples in } \omega^{<\omega} \text{ with } \bar{\eta} \equivstrOqf \bar{\nu} \}.
\]
Then for any parameter set $C$ we have that $\models \pi_\s((a_\eta)_{\eta \in \omega^{<\omega}}, C)$ iff $(a_\eta)_{\eta \in \omega^{<\omega}}$ is s-indiscernible over $C$, and similarly for $\pi_\str$ and $\pi_{\str_0}$ and being str-indiscernible and being str$_0$-indiscernible over $C$.
\end{proposition}
\begin{proof}
For $\pi_\s$ this is \cite[Corollary 5.7]{dobrowolski_kim-independence_2022}. We prove the case for $\pi_\str$ using a similar argument, and the case for $\pi_{\str_0}$ uses the exact same argument as for $\pi_\str$.

If $\models \pi_\str((a_\eta)_{\eta \in \omega^{<\omega}}, C)$ then for any finite tuples $\bar{\eta}$, $\bar{\nu}$ in $\omega^{<\omega}$ with $\bar{\eta} \equivstrqf \bar{\nu}$ we have that $\d_C(a_{\bar{\eta}}, a_{\bar{\nu}}) \leq 2$, and so in particular $a_{\bar{\eta}} \equiv_C a_{\bar{\nu}}$.

Conversely, suppose that $(a_\eta)_{\eta \in \omega^{<\omega}}$ is str-indiscernible over $C$. Let $\bar{\eta}$, $\bar{\nu}$ be finite tuples in $\omega^{<\omega}$ such that $\bar{\eta} \equivstrqf \bar{\nu}$. We may assume that both $\bar{\eta}$ and $\bar{\nu}$ are closed under meets. As $\bar{\eta}$ and $\bar{\nu}$ are both finite, there is some $k < \omega$ such that they are both contained in $k^{< k}$. Write $\bar{\eta} = (\eta_1, \ldots, \eta_n)$. For $i < \omega$ and $1 \leq j \leq n$ we define $\chi^j_i = \langle k \rangle^{(i+1)k} {}^\frown  \eta_j$, and write $\bar{\chi}_i = (\chi^1_i, \ldots, \chi^n_i)$. One straightforwardly verifies that $\bar{\eta}, \bar{\chi}_0, \bar{\chi}_1, \ldots$ and $\bar{\nu}, \bar{\chi}_0, \bar{\chi}_1, \ldots$ are indiscernible sequences with respect to quantifier-free $\L_\str$-formulas. Then $\d_C(a_{\bar{\eta}}, a_{\bar{\nu}}) \leq 2$ follows from str-indiscernibility of $(a_\eta)_{\eta \in \omega^{<\omega}}$.
\end{proof}
\begin{theorem}[s-modelling, {\cite[Proposition 5.8]{dobrowolski_kim-independence_2022}}]
\thlabel{s-modelling}
Let $T$ be a thick theory. Let $(a_\eta)_{\eta \in \omega^{<\omega}}$ be a tree of tuples and let $C$ be any set of parameters, then there is a tree $(b_\eta)_{\eta \in \omega^{<\omega}}$ that is s-indiscernible over $C$ and $\EM_\s$-based on $(a_\eta)_{\eta \in \omega^{<\omega}}$ over $C$.
\end{theorem}
\begin{proof}
The cited result \cite[Proposition 5.8]{dobrowolski_kim-independence_2022} only states the above theorem for trees of finite height (i.e., trees indexed by $\omega^{\leq k}$ for some $k < \omega$). However, type-definability of s-indiscernibility in thick theories is also established there \cite[Corollary 5.7]{dobrowolski_kim-independence_2022}. So the statement here is really just an easy application of compactness, and so we still attribute it to \cite{dobrowolski_kim-independence_2022}.
\end{proof}
Using the $<_\len$ relation in $\L_\str$ we get the following fact.
\begin{fact}
\thlabel{s-qf-type-is-str-qf-type-plus-levels}
Let $\bar{\eta}$ and $\bar{\nu}$ be finite meet-closed tuples in $\alpha^{<\beta}$, then $\bar{\eta} \equivsqf \bar{\nu}$ iff $\bar{\eta} \equivstrqf \bar{\nu}$ and $\ell(\bar{\eta}) = \ell(\bar{\nu})$.
\end{fact}
\begin{theorem}
\thlabel{str-basing}
Let $C$ be any parameter set, $\kappa$ be any cardinal, and let $\lambda = \beth_{(2^{|T|+|C|+\kappa})^+}$. Given any tree $(a_\eta)_{\eta \in \omega^{<\lambda}}$ of $\kappa$-tuples that is s-indiscernible over $C$, there is a tree $(b_\eta)_{\eta \in \omega^{<\omega}}$ that is str-indiscernible over $C$ str-based on $(a_\eta)_{\eta \in \omega^{<\lambda}}$ over $C$.
\end{theorem}
The proof of this theorem relies on the Erd\H{o}s-Rado theorem, which we will recall for the reader's convenience. For $n < \omega$ and cardinals $\kappa, \lambda, \mu$ the notation $\kappa \to (\lambda)^n_\mu$ means that for any function $f: [\kappa]^n \to \mu$ we can find $X \subseteq \kappa$ with $|X| = \lambda$ such that $f$ is constant on $[X]^n$, where $[\kappa]^n$ and $[X]^n$ are the sets of subsets of size $n$ of $\kappa$ and $X$ respectively.
\begin{fact}[Erd\H{o}s-Rado]
\thlabel{erdos-rado}
For any $n < \omega$ and any infinite cardinal $\mu$ we have that
\[
\beth^+_n(\mu) \to (\mu^+)^{n+1}_\mu.
\]
\end{fact}
\begin{proof}
Let $S$ be the set of types over $C$ in $\max(\kappa, \omega)$ many variables. Then $\lambda$ has the following properties:
\begin{enumerate}[label=(\roman*)]
\item $\lambda$ is a limit cardinal with $\cf(\lambda) > |S|$,
\item for all $\mu < \lambda$ and $n < \omega$ there is $\mu' < \lambda$ such that $\mu' \to (\mu)^n_{|S|}$.
\end{enumerate}
Property (i) is immediate, and (ii) follows from Erd\H{o}s-Rado (\thref{erdos-rado}).

For any tuple $\bar{\eta}$ in $\omega^{<\omega}$ and any type $p(x_{\bar{\eta}})$, we write
\[
\Sigma_{p, \bar{\eta}}((x_\nu)_{\nu \in \omega^{<\omega}}) = \bigcup \{ p(x_{\bar{\nu}}) : \bar{\nu} \equivstrqf \bar{\eta} \}
\]
for the partial type expressing that in the tree $(x_\nu)_{\nu \in \omega^{<\omega}}$ any tuple indexed by something with the same str-quantifier-free type as $\bar{\eta}$ has type $p$.

There are countably many quantifier-free $\L_\str$-types in finitely many variables. Enumerate the ones that are the type of a meet-closed tuple in $\omega^{<\omega}$ as $(\Delta_i)_{i < \omega}$. For $i < \omega$ we let $\bar{\eta}_i$ be a tuple in $\omega^{<\omega}$ satisfying $\Delta_i$. By induction we will construct types $(p_i(x_{\bar{\eta}_i}))_{i < \omega}$ over $C$ such that:
\begin{enumerate}[label=(\arabic*)]
\item $\bigcup_{j \leq i} \Sigma_{p_j, \bar{\eta}_j}$ is consistent,
\item for every $\mu < \lambda$ there is $I \subseteq \lambda$ with $|I| = \mu$ such that for all $j \leq i$ we have that whenever $\bar{\nu} \equivstrqf \bar{\eta}_j$ with $\ell(\bar{\nu}) \subseteq I$ then $\models p_j(a_{\bar{\nu}})$.
\end{enumerate}
Suppose that $(p_j)_{j < i}$ has been constructed, we construct $p_i$. Set $n = |\ell(\bar{\eta}_i)|$. We define $f: [\lambda]^n \to S$ as follows. For $E \in [\lambda]^n$ we let $\bar{\eta}_E$ be some tuple in $\omega^{<\lambda}$ such that $\bar{\eta}_E \equivstrqf \bar{\eta}_i$ and $\ell(\bar{\eta}_E) = E$. Then we set
\[
f(E) = \tp(a_{\bar{\eta}_E}/C).
\]
Let now $\mu < \lambda$ be arbitrary. By (ii) there is $\mu' < \lambda$ such that $\mu' \to (\mu)^n_{|S|}$. By (2) there is $I \subseteq \lambda$ with $|I| = \mu'$ such that for all $j < i$ we have that whenever $\bar{\nu} \equivstrqf \bar{\eta}_j$ with $\ell(\bar{\nu}) \subseteq I$ then $\models p_j(a_{\bar{\nu}})$ (in case $i = 0$ we just take $I = \mu'$). We apply $\mu' \to (\mu)^n_{|S|}$ to the restriction of $f$ to $I$ to find $I_\mu \subseteq I$ with $|I_\mu| = \mu$ such that $f$ is constant on $[I_\mu]^n$. We write $q_\mu$ for the constant value of $f$ on $[I_\mu]^n$, that is $q_\mu = f(E)$, where $E \in [I_\mu]^n$.

As $\mu < \lambda$ was arbitrary, there is such an $I_\mu$ and $q_\mu$ for every $\mu < \lambda$. By (i) we then find a cofinal subset $J \subseteq \lambda$ of cardinals such that $q_\mu = q_{\mu'}$ for any $\mu, \mu' \in J$. Set $p_i = q_\mu$, where $\mu \in J$. For part (2) of the induction hypothesis we note that for every $\mu < \lambda$ there is $\mu' \in J$ with $\mu < \mu'$. We verify (2) for $I_{\mu'}$ (technically we would need to take a $\mu$-sized subset, but that we can clearly do). For $j < i$ the required property follows by construction of $I_{\mu'}$, in particular because it is a subset of what we called $I$ in its construction. For $j = i$ the required property follows because $f$ is constant on $[I_{\mu'}]^n$, together with s-indiscernibility of $(a_\eta)_{\eta \in \omega^{<\lambda}}$ and \thref{s-qf-type-is-str-qf-type-plus-levels}. Then part (1) follows from part (2): let $\mu \in J$ be any infinite cardinal and let $L$ be the first $\omega$ elements of $I_\mu$, then $\bigcup_{j \leq i} \Sigma_{p_j, \bar{\eta}_j}$ is realised by the restriction $(a_\eta)_{\eta \in \omega^{<\lambda} {\upharpoonright_L}}$ to levels $L$ (see \thref{trees-examples}(iii)).

This finishes the inductive construction of $(p_i(x_{\bar{\eta}_i}))_{i < \omega}$. Set
\[
\Sigma((x_\eta)_{\eta \in \omega^{<\omega}}) = \bigcup_{i < \omega} \Sigma_{p_i, \bar{\eta}_i}.
\]
By (1) this is consistent, so we let $(b_\eta)_{\eta \in \omega^{<\omega}}$ be a realisation, which has the following properties:
\begin{enumerate}[label=(\alph*)]
\item for any finite tuple $\bar{\eta}$ in $\omega^{<\omega}$ there is some $i$ such that $\cl_\wedge(\bar{\eta}) \equivstrqf \bar{\eta}_i$,
\item if $\cl_\wedge(\bar{\eta}) \equivstrqf \bar{\eta}_i$ then $b_{\cl_\wedge(\bar{\eta})} \models p_i$.
\end{enumerate}
It follows that $(b_\eta)_{\eta \in \omega^{<\omega}}$ is str-indiscernible over $C$. Let $\bar{\eta} \equivstrqf \bar{\nu}$ be finite tuples in $\omega^{<\omega}$. Then $\cl_\wedge(\bar{\eta}) \equivstrqf \cl_\wedge(\bar{\nu})$. By (a) there is $i < \omega$ such that $\cl_\wedge(\bar{\eta}) \equivstrqf \bar{\eta}_i$. Then by (b) we get that $b_{\cl_\wedge(\bar{\eta})} \models p_i$ and $b_{\cl_\wedge(\bar{\nu})} \models p_i$, so since $p_i$ is a type over $C$ we get $b_{\bar{\eta}} \equiv_C b_{\bar{\nu}}$ after restricting the types. Finally, $(b_\eta)_{\eta \in \omega^{<\omega}}$ is str-based on $(a_\eta)_{\eta \in \omega^{<\lambda}}$: for any finite tuple $\bar{\eta}$ in $\omega^{<\omega}$ we have by (a) that there is $i < \omega$ with $\cl_\wedge(\bar{\eta}) = \bar{\eta}_i$. By construction $p_i$ is realised by $a_{\bar{\nu}'}$ for some $\bar{\nu}'$ in $\omega^{<\lambda}$ with $\bar{\nu}' \equivstrqf \bar{\eta}_i$. Using the fact that $\bar{\nu}' \equivstrqf \cl_\wedge(\bar{\eta})$ we find $\bar{\nu} \subseteq \bar{\nu}'$ such that $\bar{\nu} \equivstrqf \bar{\eta}$. By (b) we have that $\tp(b_{\cl_\wedge(\bar{\eta})}/C) = p_i$, so restricting types yields $b_{\bar{\eta}} \equiv_C a_{\bar{\nu}}$, as required.
\end{proof}
\begin{repeated-theorem}[\thref{str-modelling}]
Let $T$ be a thick theory. Let $(a_\eta)_{\eta \in \omega^{<\omega}}$ be a tree of tuples of the same length and let $C$ be any set of parameters, then there is a tree $(b_\eta)_{\eta \in \omega^{<\omega}}$ that is str-indiscernible over $C$ and $\EM_\str$-based on $(a_\eta)_{\eta \in \omega^{<\omega}}$ over $C$.
\end{repeated-theorem}
\begin{proof}
By \thref{s-modelling} we find $(a'_\eta)_{\eta \in \omega^{<\omega}}$ that is s-indiscernible over $C$ and is $\EM_\s$-based on $(a_\eta)_{\eta \in \omega^{<\omega}}$ over $C$. Let $\lambda$ be the cardinal in \thref{str-basing}. Using \thref{s-str-str0-indiscernible-type-definable} we can write down a partial type $\Sigma$ for a tree $(a''_\eta)_{\eta \in \omega^{<\lambda}}$ that is s-indiscernible over $C$ and $\EM_\str$-based on $(a'_\eta)_{\eta \in \omega^{<\omega}}$ over $C$. Using $(a'_\eta)_{\eta \in \omega^{<\omega}}$, and renaming levels whenever needed, we see that $\Sigma$ is finitely satisfiable. So we find our tree $(a''_\eta)_{\eta \in \omega^{<\lambda}}$ that is  s-indiscernible over $C$ and $\EM_\str$-based on $(a'_\eta)_{\eta \in \omega^{<\omega}}$ over $C$. We apply \thref{str-basing} to $(a''_\eta)_{\eta \in \omega^{<\lambda}}$ to find $(b_\eta)_{\eta \in \omega^{<\omega}}$ that is str-indiscernible over $C$ and str-based on $(a''_\eta)_{\eta \in \omega^{<\lambda}}$ over $C$. Being str-based and $\EM_\s$-based both imply being $\EM_\str$-based, and being $\EM_\str$-based is transitive, so $(b_\eta)_{\eta \in \omega^{<\omega}}$ is $\EM_\str$-based on $(a_\eta)_{\eta \in \omega^{<\omega}}$ over $C$.
\end{proof}
The proof strategies in \thref{str-basing,str-modelling} are very similar to the case for indiscernible sequences. The use of the Erd\H{o}s-Rado theorem in \thref{str-basing} is very similar how one constructs an indiscernible sequence based on a very long sequence (see \thref{indiscernible-sequence-basing-on} and the reference there). Then \thref{str-modelling} is similar to \thref{sequence-em-modelling}: we use compactness to stretch the input and then apply the previous result that uses the Erd\H{o}s-Rado theorem. One key difference though is that \thref{str-basing} requires the input to already be s-indiscernible. This is why we assume thickness in \thref{str-modelling}. That way we can use type-definability of s-indiscernibility to guarantee that the stretched input remains s-indiscernible.
\begin{theorem}
\thlabel{str0-basing}
Let $C$ be any parameter set. Given any tree $(a_\eta)_{\eta \in \omega^{<\omega}}$ that is str-indiscernible over $C$, there is a tree $(b_\eta)_{\eta \in \omega^{<\omega}}$ that is str$_0$-indiscernible over $C$ and str$_0$-based on $(a_\eta)_{\eta \in \omega^{<\omega}}$ over $C$.
\end{theorem}
\begin{proof}
Fix $1 \leq k < \omega$. By induction on $m < \omega$ we define $f_k^m: k^{\leq m} \to \omega^{<\omega}$ and $l_k^m < \omega$ as in \cite[Claim A.7]{scow_indiscernibles_2015} (which in turn is based on \cite[page 142]{dzamonja_lhd-maximality_2004}):
\begin{align*}
f_k^m(\emptyset) &= \emptyset &\text{for all } m < \omega, \\
l_k^m &= \max \{ \ell(f_k^m(\eta)) + 1 : \eta \in k^{\leq m} \}, \\
f_k^{m+1}(\langle i \rangle^\frown \eta) &= \langle i \rangle^\frown \langle 0 \rangle^{(i+1)l_k^m} {}^\frown f_k^m(\eta).
\end{align*}
Then for all $k,m < \omega$:
\begin{center}
(*) $f_k^m$ is an $\L_{\str_0}$-embedding such that $\eta <_\lex \nu$ implies $f_k^m(\eta) <_\len f_k^m(\nu)$.
\end{center}
Let $\bar{\eta}$ be a finite tuple in $\omega^{< \omega}$ and let $k, m < \omega$ be such that $\bar{\eta}$ is contained in $k^{\leq m}$. Then we assign the type $p_{\bar{\eta}}(x_{\bar{\eta}}) = \tp(a_{f_m^k(\bar{\eta})}/C)$ to $\bar{\eta}$. By (*) and str-indiscernibility, the type $p_{\bar{\eta}}$ does not depend on $k$ or $m$ and whenever $\bar{\eta} \equivstrOqf \bar{\nu}$ then $p_{\bar{\eta}} = p_{\bar{\nu}}$ (after renaming variables). Define
\[
\Sigma((x_\eta)_{\eta \in \omega^{<\omega}}) = \bigcup \{ p_{\bar{\eta}} : \bar{\eta} \text{ is a finite tuple in } \omega^{<\omega} \}.
\]
Let $\Sigma_0$ be any finite part of $\Sigma$, and let $k,m < \omega$ be such that the variables appearing in $\Sigma_0$ are contained in $(x_\eta)_{\eta \in k^{\leq m}}$. Then $\Sigma_0$ is realised by $(a'_\eta)_{\eta \in k^{\leq m}}$ where $a'_\eta = a_{f_k^m(\eta)}$ for all $\eta \in k^{\leq m}$. So by compactness we find a realisation $(b_\eta)_{\eta \in \omega^{<\omega}}$ of $\Sigma$, which is the tree we needed to construct. Indeed, let $\bar{\eta}$ be any finite tuple in $\omega^{<\omega}$ and let $k,m < \omega$ be such that $\bar{\eta}$ is contained in $k^{\leq m}$. Then $\tp(b_{\bar{\eta}}/C) = p_{\bar{\eta}} = \tp(a_{f_k^m(\bar{\eta})}/C)$, and so because $f_k^m(\bar{\eta}) \equivstrOqf \bar{\eta}$ this shows that $(b_\eta)_{\eta \in \omega^{<\omega}}$ is str$_0$-based on $(a_\eta)_{\eta \in \omega^{<\omega}}$ over $C$. Let now $\bar{\nu}$ be such that $\bar{\nu} \equivstrOqf \bar{\eta}$, then $\tp(b_{\bar{\eta}}/C) = p_{\bar{\eta}} = p_{\bar{\nu}} = \tp(b_{\bar{\nu}}/C)$ and we have established str$_0$-indiscernibility over $C$.
\end{proof}
\begin{repeated-theorem}[\thref{str0-modelling}]
Let $T$ be a thick theory. Let $(a_\eta)_{\eta \in \omega^{<\omega}}$ be a tree of tuples of the same length and let $C$ be any set of parameters, then there is a tree $(b_\eta)_{\eta \in \omega^{<\omega}}$ that is str$_0$-indiscernible over $C$ and $\EM_{\str_0}$-based on $(a_\eta)_{\eta \in \omega^{<\omega}}$ over $C$.
\end{repeated-theorem}
\begin{proof}
By \thref{str-modelling} there is a tree $(a'_\eta)_{\eta \in \omega^{<\omega}}$ that is str-indiscernible over $C$ and $\EM_\str$-based on $(a_\eta)_{\eta \in \omega^{<\omega}}$ over $C$. By \thref{str0-basing} there is then a tree $(b_\eta)_{\eta \in \omega^{<\omega}}$ that is str$_0$-indiscernible over $C$ and str$_0$-based on $(a'_\eta)_{\eta \in \omega^{<\omega}}$ over $C$. Being $\EM_\str$-based and being str$_0$-based both imply being $\EM_{\str_0}$-based, and being $\EM_{\str_0}$-based is transitive, so we conclude that $(b_\eta)_{\eta \in \omega^{<\omega}}$ is the required tree.
\end{proof}

\section{The array modelling theorem}
\begin{definition}[{\cite[Definition 5.4]{kim_tree_2014}}]
\thlabel{array-language}
We define the following structure on $\omega \times \omega$, where $(i, j), (s, t) \in \omega \times \omega$:
\begin{itemize}
\item $(i, j) <_1 (s, t)$ iff $i < s$,
\item $(i, j) <_2 (s, t)$ iff $i = s$ and $j <  t$.
\end{itemize}
We define the \emph{array language} to be $\L_\ar = \{ <_1, <_2 \}$ and we call the structure $\omega \times \omega$ an \emph{array}. We abbreviate notation involving $\L_\ar$ in a similar way as described at the end of \thref{tree-language}, replacing the ``s'' there by ``ar'' or ``array''.
\end{definition}
\begin{repeated-theorem}[\thref{array-modelling}]
Let $T$ be a thick theory. Let $(a_{i,j})_{i,j < \omega}$ be an array of tuples of the same length and let $C$ be any set of parameters, then there is an array $(b_{i,j})_{i,j < \omega}$ that is array-indiscernible over $C$ and $\EM_{\ar}$-based on $(a_{i,j})_{i,j < \omega}$ over $C$.
\end{repeated-theorem}
\begin{proof}
Let $J$ be the $<_\lex$-order type of $\omega^{<\omega}$ and, using compactness, let $(a'_{i, j})_{i < \omega, j \in J}$ be an array that is $\EM_\ar$-based on $(a_{i,j})_{i,j < \omega}$ over $C$. Here $\omega \times J$ carries the expected $\L_\ar$-structure: $(i, j) <_1 (s, t)$ iff $i < s$, and $(i, j) <_2 (s, t)$ iff $i = s$ and $j < t$ in the order on $J$. Let $f': \omega^{<\omega} \to J$ be the $<_\lex$-order isomorphism and define $f: \omega^{<\omega} \to \omega \times J$ as $f(\eta) = (\ell(\eta), f'(\eta))$. Then $f$ is an injection such that for any $\eta, \nu \in \omega^{<\omega}$:
\begin{enumerate}[label=(\roman*)]
\item $\eta <_\len \nu$ implies $f(\eta) <_1 f(\nu)$,
\item $\ell(\eta) = \ell(\nu)$ and $\eta <_\lex \nu$ implies $f(\eta) <_2 f(\nu)$.
\end{enumerate}
So in particular, $f$ is qftp-respecting, where $\omega^{< \omega}$ is considered as an str-tree. Define a tree $(a^*_\eta)_{\eta \in \omega^{<\omega}}$ by $a^*_\eta = a'_{f(\eta)}$. By \thref{str-modelling} we find a tree $(b^*_\eta)_{\eta \in \omega^{<\omega}}$ that is str-indiscernible over $C$ and $\EM_\str$-based on $(a^*_\eta)_{\eta \in \omega^{<\omega}}$. Define $g: \omega \times \omega \to \omega^{<\omega}$ by $g(i, j) = {\langle 0 \rangle^{2i}}^\frown \langle j + 1 \rangle$ and define an array $(b_{i, j})_{i, j < \omega}$ by $b_{i, j} = b^*_{g(i, j)}$. We claim that $(b_{i, j})_{i,j < \omega}$ is the required array.

First, we note that $g$ is qftp-respecting. Indeed, for any $(i,j), (s, t) \in \omega \times \omega$ we have that:
\begin{itemize}
\item $(i, j) <_1 (s, t)$ implies $g((i, j)) <_\len g((s, t))$ and $g((s, t)) <_\lex g((i, j))$,
\item $(i, j) <_2 (s, t)$ implies $g((i, j)) <_\lex g((s, t))$,
\item $g((i, j)) \unlhd g((s, t))$ iff $(i, j) = (s, t)$.
\end{itemize}
That leaves the relations between any meets in the image of $g$ to be checked, but this is also straightforward using the fact that the meet of a finite number of nodes in $\omega^{<\omega}$ can be written as the meet of two of those nodes, together with the fact that $g((i, j)) \wedge g((s, t)) = \langle 0 \rangle^{2 \min(i, s)}$ (unless $(i, j) = (s, t)$, of course). We can thus apply the re-indexing lemma, \thref{reindexing-lemma}(i), and get that $(b_{i,j})_{i,j < \omega}$ is array-indiscernible over $C$ because $(b^*_\eta)_{\eta \in \omega^{<\omega}}$ is str-indiscernible over $C$.

We will again use the re-indexing lemma, \thref{reindexing-lemma}(iii), to conclude that $(b_{i,j})_{i,j < \omega}$ is $\EM_\ar$-based on $(a'_{i,j})_{i < \omega, j \in J}$ over $C$. From this the result then follows because $(a'_{i,j})_{i < \omega, j \in J}$ is $\EM_\ar$-based on $(a_{i,j})_{i,j < \omega}$ over $C$. So we only need to verify that for any finite tuple $\bar{\eta}$ in $\omega \times \omega$ we have that $\bar{\eta} \equivarqf fg(\bar{\eta})$. Indeed, let $(i, j), (s, t) \in \omega \times \omega$, then
\begin{itemize}
\item $f$ and $g$ are both injective functions, so equality is preserved and reflected;
\item if $(i, j) <_1 (s, t)$ then $g((i, j)) <_\len g((s, t))$, and so $fg((i, j)) <_1 fg((s, t))$;
\item if $fg((i, j)) <_1 fg((s, t))$ then $2i + 1 = \ell(g((i, j))) < \ell(g((s, t))) = 2s + 1$ and so $(i, j) <_1 (s, t)$;
\item if $(i, j) <_2 (s, t)$ then $\ell(g((i, j))) = \ell(g((s, t)))$ and $g((i, j)) <_\lex g((s, t))$, so $fg((i, j)) <_2 fg((s, t))$;
\item if $fg((i, j)) <_2 fg((s, t))$ then $\ell(g((i, j))) = \ell(g((s, t)))$ and $f'(g((i, j))) < f'(g((s, t)))$, the latter means that $g((i, j))) <_\lex g((s, t))$ and since their lengths are the same, and so $i = s$, we must have $j+1 < t+1$ and thus $(i, j) <_2 (s, t)$.
\end{itemize}
\end{proof}
\begin{remark}
\thlabel{kim-kim-scow-embedding}
The proof of \thref{array-modelling} is based on that of \cite[Theorem 5.5]{kim_tree_2014}. However, there the existence of a non-existing embedding is claimed. In more detail, we view $\omega^{<\omega}$ as an $\L_\ar$-structure by interpreting $<_1$ as $<_\len$ and setting $\eta <_2 \nu$ iff $\ell(\eta) = \ell(\nu)$ and $\eta <_\lex \nu$. The claim is then that there is an $\L_\ar$-embedding $f: \omega^{<\omega} \to \omega \times \omega$. Such an $f$ cannot exist. Our proof fixes this by stretching the original array, so that our $f$ can play the role of this embedding, and the third author of \cite{kim_tree_2014} has a similar fix in \cite[Corollary 3.8]{scow_ramsey_2021}.

Suppose that an $f$ as above exists. Let $\eta, \nu \in \omega^{<\omega}$ be such that $\ell(\eta) = \ell(\nu)$, and let $(i, j) = f(\eta)$ and $(s, t) = f(\nu)$. Then we must have $i = s$, because $<_1$ is interpreted as $<_\len$ in the tree and $<_1$ is preserved and reflected by $f$. Let $g$ be the restriction of $f$ to the nodes of length two. Then the image of $g$ is contained in $\{n\} \times \omega$ for some $n < \omega$. So $g(\langle 1, 0 \rangle) = (n, k)$ for some $k < \omega$. However, $\langle 0, i \rangle <_\lex \langle 1, 0 \rangle$ for all $i < \omega$. We thus get that the second coordinate of $g(\langle 0, i \rangle)$ is strictly less than $k$ for all $i < \omega$, contradicting injectivity of $g$. So $f$ cannot exist.
\end{remark}
Though we have no use for it, the following may be useful in future applications.
\begin{proposition}
\thlabel{array-indiscernibility-type-definable}
Let $T$ be a thick theory. The property of being array-indiscernible is type-definable. That is, let $\pi_\ar((x_{i,j})_{i,j < \omega}, y)$ be the union of the following partial types:
\begin{itemize}
\item a partial type that expresses that $((x_{i,j})_{j < \omega})_{i < \omega}$ is $y$-indiscernible;
\item for each $i < \omega$, a partial type that expresses that $(x_{i,j})_{j < \omega}$ is indiscernible over $(y, (x_{k,j})_{k,j < \omega, k \neq i})$.
\end{itemize}
Then for any parameter set $C$ we have that $\models \pi_\ar((a_{i,j})_{i,j < \omega}, C)$ iff $(a_{i,j})_{i,j < \omega}$ is array-indiscernible over $C$.
\end{proposition}
\begin{proof}
If $(a_{i,j})_{i,j < \omega}$ is array-indiscernible over $C$ then we clearly have that:
\begin{itemize}
\item $((a_{i,j})_{j < \omega})_{i < \omega}$ is $C$-indiscernible;
\item for each $i < \omega$, $(a_{i,j})_{j < \omega}$ is indiscernible over $(C, (a_{k,j})_{k,j < \omega, k \neq i})$.
\end{itemize}
So $\models \pi_\ar((a_{i,j})_{i,j < \omega}, C)$.

For the converse, we assume $\models \pi_\ar((a_{i,j})_{i,j < \omega}, C)$ and let $(q_1, r_1), \ldots, (q_n, r_n) \in \omega \times \omega$ and $(s_1, t_1), \ldots, (s_n, t_n) \in \omega \times \omega$ be such that $(q_1, r_1) \ldots (q_n, r_n) \equivarqf (s_1, t_1) \ldots (s_n, t_n)$. For notational convenience we may assume that $(q_1, r_1) \leq_1 (q_2, r_2) \leq_1 \ldots \leq_1 (q_n, r_n)$ and (necessarily) the same for the $(s_i, t_i)$. Then because $((a_{i,j})_{j < \omega})_{i < \omega}$ is $C$-indiscernible we have that
\[
a_{q_1, r_1} \ldots a_{q_n, r_n} \equiv_C a_{s_1, r_1} \ldots a_{s_n, r_n}.
\]
Now let $1 \leq m \leq n$ be maximal such that $s_1 = s_m$ (so for all $1 \leq m' \leq m$ we have that $s_1 = s_{m'}$). Then because $(a_{s_1,j})_{j < \omega}$ is indiscernible over $(C, (a_{k,j})_{k,j < \omega, k \neq s_1})$, we have that
\[
a_{s_1, r_1} \ldots a_{s_n, r_n} \equiv_C a_{s_1, t_1} \ldots a_{s_m, t_m} a_{s_{m+1}, r_{m+1}} \ldots a_{s_n, r_n}.
\]
Repeating this process we find
\[
a_{s_1, r_1} \ldots a_{s_n, r_n} \equiv_C a_{s_1, t_1} \ldots a_{s_n, t_n},
\]
and thus $a_{q_1, r_1} \ldots a_{q_n, r_n} \equiv_C a_{s_1, t_1} \ldots a_{s_n, t_n}$, as required.
\end{proof}
\section{An application: \TPtwo}
The definition of $2$-\TPtwo in positive logic first appeared in \cite[Definition 6.1]{haykazyan_existentially_2021}. We take the version from \cite{dmitrieva_dividing_2023}, which defines $k$-\TPtwo for $k \geq 2$.
\begin{definition}[{\cite[Definition 4.5]{dmitrieva_dividing_2023}}]
\thlabel{def-tp2}
A formula $\phi(x,y)$ has the \emph{$k$-tree property of the second kind} ($k$-\TPtwo) for $k \geq 2$ if there are $(a_{i,j})_{i,j < \omega}$ and $\psi(y_1, \ldots, y_k)$ that implies $\neg \exists x (\phi(x, y_1) \wedge \ldots \wedge \phi(x, y_k))$ modulo $T$ such that:
\begin{enumerate}[label=(\roman*)]
\item for all $\sigma \in \omega^\omega$ the set $\{ \phi(x, a_{i, \sigma(i)}): i < \omega\}$ is consistent,
\item for all $i < \omega$ and $j_1 < \ldots < j_k < \omega$ we have that $\models \psi(a_{i,j_1}, \ldots, a_{i,j_k})$.
\end{enumerate}
We say that a theory $T$ has $k$-\TPtwo if some formula has $k$-\TPtwo.
\end{definition}
\begin{lemma}
\thlabel{indiscernible-tp2}
Let $T$ be a thick theory. A formula $\phi(x, y)$ has $k$-\TPtwo for $k \geq 2$ if and only if there is an array-indiscernible array $(a_{i,j})_{i,j < \omega}$ such that
\begin{enumerate}[label=(\roman*)]
\item $\{ \phi(x, a_{i, 0}) : i < \omega\}$ is consistent,
\item $\{ \phi(x, a_{0, i}) : i < \omega\}$ is $k$-inconsistent.
\end{enumerate}
\end{lemma}
\begin{proof}
If $\phi(x, y)$ has $k$-\TPtwo as witnessed by $(a'_{i,j})_{i,j < \omega}$ and $\psi(y_1, \ldots, y_k)$ then we can use \thref{array-modelling} to $\EM_\ar$-base an array-indiscernible array $(a_{i,j})_{i, j < \omega}$ on $(a'_{i,j})_{i,j < \omega}$. Then points (i) and (ii) in \thref{def-tp2} are captured by the $\EM_\ar$-type of $(a'_{i,j})_{i,j < \omega}$ and respectively yield points (i) and (ii) for $(a_{i,j})_{i, j < \omega}$ as above.

Conversely, let $(a_{i,j})_{i, j < \omega}$ be as above. Then $((i, \sigma(i)))_{i < \omega} \equivarqf ((i,0))_{i < \omega}$ for any $\sigma \in \omega^\omega$, so by (i) and array-indiscernibility $\{ \phi(x, a_{i, \sigma(i)}) : i < \omega\}$ is consistent. By (ii) we have that $\not \models \exists x(\phi(x, a_{0,0}) \wedge \ldots \wedge \phi(x, a_{0,k-1}))$. Let $\psi(y_1, \ldots, y_k)$ be such that it implies $\neg \exists x(\phi(x, y_1) \wedge \ldots \wedge \phi(x, y_k))$ modulo $T$ and $\models \psi(a_{0,0}, \ldots, a_{0,k-1})$. Then by array-indiscernibility $\models \psi(a_{i,j_1}, \ldots, a_{i,j_k})$ for all $i < \omega$ and $j_1 < \ldots < j_k < \omega$. We conclude that $(a_{i,j})_{i,j < \omega}$ and $\psi$ witness $k$-\TPtwo for $\phi(x, y)$.
\end{proof}
\begin{repeated-theorem}[\thref{k-tp2-iff-2-tp2}]
Let $T$ be a thick theory. If $\phi(x, y)$ has $k$-\TPtwo for some $k \geq 2$ then some conjunction $\bigwedge_{i = 1}^n \phi(x, y_i)$ has $2$-\TPtwo. Hence, $T$ has $k$-\TPtwo for some $k \geq 2$ iff $T$ has $2$-\TPtwo.
\end{repeated-theorem}
\begin{proof}
Based on \cite[Proposition 5.7]{kim_tree_2014}. We prove this by induction on $k \geq 2$. The base case, $k = 2$, is trivial. So assume that the theorem holds for $2, \ldots, k - 1$ and suppose that $\phi(x, y)$ has $k$-\TPtwo. By \thref{indiscernible-tp2} there is an array-indiscernible array $(a_{i,j})_{i,j < \omega}$ such that
\begin{enumerate}[label=(\roman*)]
\item $\{ \phi(x, a_{i, 0}) : i < \omega\}$ is consistent,
\item $\{ \phi(x, a_{0, i}) : i < \omega\}$ is $k$-inconsistent.
\end{enumerate}
We finish the induction step, and thus the proof, by considering two cases.
\begin{itemize}
\item The case where $\{ \phi(x, a_{i, 0}),\phi(x, a_{i, 1}) : i < \omega\}$ is consistent. Define an array $(b_{i,j})_{i,j < \omega}$ by $b_{i,j} = (a_{i,2j}, a_{i,2j+1})$ for all $i,j < \omega$, then $(b_{i,j})_{i,j < \omega}$ is array-indiscernible. Let $\phi'(x, y_0, y_1)$ be $\phi(x, y_0) \wedge \phi(x, y_1)$. Then by assumption $\{ \phi'(x, b_{i, 0}) : i < \omega\}$ is consistent, while $\{ \phi'(x, b_{0, i}) : i < \omega\}$ is $\lceil k / 2 \rceil$-inconsistent. So by \thref{indiscernible-tp2} we have that $\phi'(x; y_0, y_1)$ has $\lceil k / 2 \rceil$-\TPtwo, and we conclude this case by using the induction hypothesis.
\item The case where $\{ \phi(x, a_{i, 0}),\phi(x, a_{i, 1}) : i < \omega\}$ is inconsistent. By compactness there is then $n < \omega$ such that $\{ \phi(x, a_{i, 0}),\phi(x, a_{i, 1}) : i < n\}$ is inconsistent. Define an array $(b_{i,j})_{i,j < \omega}$ by $b_{i,j} = (a_{ni, j}, \ldots, a_{ni+n-1,j})$, then $(b_{i,j})_{i,j < \omega}$ is array-indiscernible. Let $\phi'(x, y_0, \ldots, y_{n-1})$ be $\phi(x, y_0) \wedge \ldots \wedge \phi(x, y_{n-1})$. Then $\{ \phi'(x, b_{i, 0}) : i < \omega\}$ is consistent, while by assumption $\{ \phi'(x, b_{0, i}) : i < \omega\}$ is $2$-inconsistent. So by \thref{indiscernible-tp2} we have that $\phi'(x; y_0, \ldots, y_{n-1})$ has $2$-\TPtwo, completing the induction step for this case.
\end{itemize}
\end{proof}

\bibliographystyle{alpha}
\bibliography{bibfile}


\end{document}